\documentclass[12pt]{amsart}

\usepackage{mathrsfs}
\usepackage{amsfonts}
\usepackage{amssymb}
\usepackage{amsfonts, amscd, amsmath, mathrsfs, amssymb, amsthm, amsxtra, bbding, epsfig, graphicx, latexsym, url}
\usepackage[papersize={8.3in,11.8in},textwidth=16.8cm,textheight=23.2cm,centering]{geometry}
\usepackage{enumerate}
\usepackage{caption} 

\usepackage{graphicx}
\usepackage{tikz}
\usepackage{float}

\usepackage{xcolor}
\definecolor{cite}{rgb}{0.00,0.60,1.00}
\definecolor{url}{rgb}{1.00,0.10,0.80}
\definecolor{link}{rgb}{0.00,0.00,1.00}
\usepackage[colorlinks,linkcolor=link,urlcolor=url,citecolor=cite,pagebackref,breaklinks]{hyperref}

\hypersetup{
pdfstartpage=1,
pdfstartview=FitH}

\DeclareFontFamily{U}{mathx}{\hyphenchar\font45}
\DeclareFontShape{U}{mathx}{m}{n}{
      <5> <6> <7> <8> <9> <10>
      <10.95> <12> <14.4> <17.28> <20.74> <24.88>
      mathx10
      }{}
\DeclareSymbolFont{mathx}{U}{mathx}{m}{n}
\DeclareMathAccent{\widecheck}{\mathalpha}{mathx}{"71}

\numberwithin{equation}{section}

\allowdisplaybreaks

\newtheorem{theorem}{Theorem}[section]
\newtheorem{lemma}{Lemma}[section]

\newtheorem{definition}{Definition}[section]
\newtheorem{proposition}{Proposition}[section]

\newtheorem{conjecture}{Conjecture}[section]

\makeatletter
\newcounter{roem}
\renewcommand{\theroem}{\Roman{roem}}

\newcommand{\c@org@eq}{}
\let\c@org@eq\c@equation
\newcommand{\org@theeq}{}
\let\org@theeq\theequation

\newcommand{\setroem}{
\let\c@equation\c@roem
 \let\theequation\theroem}

\newcommand{\setarab}{
\let\c@equation\c@org@eq
\let\theequation\org@theeq}
\makeatother

\newtheorem*{claim*}{Claim}

\theoremstyle{theorem}

\newcommand{\ud}{\mathrm{d}}
\newcommand{\ue}{\mathrm{e}}

\newcommand{\sym}{\mathrm{sym}}

\newcommand{\kl}{\mathrm{Kl}}

\newcommand{\SL}{\mathrm{SL}}
\newcommand{\PSL}{\mathrm{PSL}}

\usepackage{scalerel}

\DeclareMathOperator{\Mod}{mod}

\renewcommand{\bmod}[1]{\,(\Mod{ #1})}

\newcommand{\F}{\mathbb{F}}

\newcommand{\R}{\mathbb{R}}
\newcommand{\Z}{\mathbb{Z}}

\newcommand{\bdelta}{\boldsymbol{\delta}}
\newcommand{\blambda}{\boldsymbol{\lambda}}

\newcommand{\bC}{\mathbf{C}}

\newcommand{\bZ}{\mathbf{Z}}

\newcommand{\cA}{\mathcal{A}}
\newcommand{\cB}{\mathcal{B}}

\newcommand{\cE}{\mathcal{E}}
\newcommand{\cF}{\mathcal{F}}
\newcommand{\cG}{\mathcal{G}}

\newcommand{\cM}{\mathcal{M}}
\newcommand{\cN}{\mathcal{N}}
\newcommand{\cP}{\mathcal{P}}

\newcommand{\cR}{\mathcal{R}}

\newcommand{\fP}{\mathfrak{P}}
\newcommand{\fS}{\mathfrak{S}}
\newcommand{\fX}{\mathfrak{X}}

\def\ge{\geqslant}
\def\geq{\geqslant}

\usepackage{graphicx}
\usepackage{tikz}

\begin{document}

\title{Horizontal moments of Kloosterman sums and a density version of sieves}

\author{Ping Xi}

\address{School of Mathematics and Statistics, Xi'an Jiaotong University, Xi'an 710049, P. R. China}
\address{Institute of Pure Mathematics, Xi'an Jiaotong University, Xi'an 710049, P. R. China}

\email{pingxi.cn@gmail.com, ping.xi@xjtu.edu.cn}

\subjclass[2020]{11N13, 11N36, 11T23, 11L05, 11L07}

\keywords{Kloosterman sums, moments, Sato--Tate distribution}

\begin{abstract} 
Denote by $\kl(a,n)$ the normalized Kloosterman sum modulo $n$. In this paper, we study the moments of $|\kl(1,n)|$ on average over $n\leqslant X$, and show that
\begin{align*}
\sum_{n\leqslant X}|\kl(1,n)|\asymp X(\log X)^{\frac{8}{3\pi}-1}.
\end{align*}
The lower bound is proven by considering a subfamily of moduli which factorize in a nice way such that vertical Sato--Tate distributions of Kloosterman sums (as variants of Katz) apply. The upper bound is based on a new density version of upper bound sieves, combining with averages of divisor functions in arithmetic progressions and a similar vertical Sato--Tate distribution.
\end{abstract}

\maketitle

\setcounter{tocdepth}{1}


\section{Introduction and main results}
\label{sec:results}

\subsection{Background and main results}
We begin with the normalized Kloosterman sum
\[\kl(a,n)=\frac{1}{\sqrt{n}}~\sideset{}{^*}\sum_{u\bmod n}\ue\Big(\frac{au+\overline{u}}{n}\Big)\]
defined for all $n\in\bZ^+$ and $a\in\bZ,$ where $u\overline{u}\equiv1\bmod n.$ The celebrated Weil's bound asserts that $|\kl(a,p)|\leqslant2$ for all $p\in\cP,$ where $\cP$ denotes the set of all primes, and we identify the finite field $\F_p$ by $\{0,1,\cdots,p-1\}.$
The constant 2 is optimal in the sense that for any $\varepsilon>0,$ there exists $\gg p$ many $a\in\F_p^\times$ such that $|\kl(a,p)|>2-\varepsilon.$ This is an immediate consequence of vertical Sato--Tate theorem of Katz \cite{Ka88}, which asserts that the set $\{\kl(a,p):a\in\F_p^\times\}\subseteq[-2,2]$ becomes equidistributed with respect to the Sato--Tate measure $\frac{1}{2\pi}\sqrt{4-t^2}\ud t$, as $p\rightarrow+\infty.$ In the horizontal direction, however, as $a$ fixed and $p$ varies over primes the statistics of $\kl(a,p)$ is much less understood. Let us recall the following well-known conjecture of Katz \cite{Ka80}.

\begin{conjecture}[Sato--Tate conjecture for Kloosterman sums]\label{conj:KatzST}
Let $a\neq0$ be a fixed integer. Then the set $\{\kl(a,p):p\in\cP, p\nmid a\}\subseteq[-2,2]$ equidistributes with respect to the Sato--Tate measure $\frac{1}{2\pi}\sqrt{4-t^2}\ud t$.
\end{conjecture}

Instead of prime moduli, an easier but quite related question is to consider the behaviours of $\kl(a,n)$ as $n$ runs over consecutive integers. The Linnik--Selberg conjecture asserts that
\begin{align}\label{eq:LinnikSelberg}
\sum_{n\leqslant X}\kl(1,n)\ll X^{\frac{1}{2}+\varepsilon}
\end{align}
holds for any $\varepsilon>0$.  Kuznetsov proved that one may take with any $\varepsilon>1/6$ in \eqref{eq:LinnikSelberg}. On the other hand, it was remarked by Serre that one should understand the truth of \eqref{eq:LinnikSelberg}, depending on the small sizes of many Kloosterman sums or the sign changes among them (recorded in \cite[Page 33]{Sa90}).
Hence this motivates us to introduce the first moment of Kloosterman sums with absolute values in the horizontal direction:
\begin{align*}
\fS(X)=\sum_{n\leqslant X}|\kl(1,n)|.
\end{align*}

A pioneering work of Fouvry and Michel \cite{FM03} shows that
\begin{equation}\label{eq:FM-upper}
\frac{X}{\log X}\exp((\log\log X)^{\delta})\ll \fS(X) \ll X(\log X)^{\frac{8}{3\pi}-1}(\log\log X)^{2-\frac{8}{3\pi}}
\end{equation}
for some $\delta>5/12$.  Their proof is based on sieving arguments and equidistributions of Kloosterman sums in the vertical direction (as variants and generalizations of Katz's original equidistribution over $\F_p^\times$).

We establish the following theorem, which has been a folklore for a long time.
\begin{theorem}\label{thm:main}
For all large $X,$ we have
\begin{align*}
\fS(X)\asymp X(\log X)^{\frac{8}{3\pi}-1}.
\end{align*}
\end{theorem}

We now explain the heuristics in Theorem \ref{thm:main}. An immediate consequence of Conjecture \ref{conj:KatzST}, which is still open, that
\begin{align*}
\sum_{p\leqslant X}|\kl(a,p)|\sim \frac{8}{3\pi}\frac{X}{\log X}
\end{align*}
for each fixed $a\neq0$. This implies that $|\kl(a,p)|$ behaves like $8/(3\pi)$ on average over primes $p$.
On the other hand, for $(n_1,n_2)=1,$ we have the twisted multiplicativity
$\kl(a,n_1n_2)=\kl(a\overline{n_2^2},n_1)\kl(a\overline{n_1^2},n_2)$
as a consequence of  the Chinese remainder theorem. Hence a courageous expectation is that 
$|\kl(a,n)|$ may behave like $(8/(3\pi))^{\omega(n)}$ on average over squarefree integers $n$, although the Kloosterman sums modulo different prime factors of $n$ should correlate quite a lot with each other. Here $\omega(n)$ counts the number of distinct prime factors of $n$. Such a crude random model of Kloosterman sums predicts that the first moment $\fS(X)$ should be of the same of order of magnitude as $\sum_{n\leqslant X}(8/(3\pi))^{\omega(n)}$, and Theorem \ref{thm:main} shows this is indeed the case.

The arguments in this paper also allow us to derive the following theorem for general positive moments, with general frequencies. Denote by $\mu$, $\varphi$ the M\"obius function and Euler's totient function, respectively, and by $\Gamma$ the classical Gamma function.

\begin{theorem}\label{thm:general}
Let  $\nu>0$ be a fixed real number and $f\in\{1,\mu^2\}$. For all large $X$ and $a,b\in\Z^+$ with $a,b\leqslant X/2,$ we have
\begin{align*}
\sum_{\substack{n\leqslant X\\ (n,ab)=1}}f(n)|\kl(\pm a\overline{b},n)|^\nu \asymp \Big(\frac{\varphi(ab)}{ab}\Big)^{c_\nu} X(\log X)^{c_\nu-1}
\end{align*}
with 
\begin{align*}
c_\nu=\frac{\Gamma(\nu+1)}{\Gamma(\frac{\nu}{2}+2)\Gamma(\frac{\nu}{2}+1)},
\end{align*}
where the constants implied by $\asymp$ depend only on $\nu.$
\end{theorem}

\subsection{Sketch of the proof} 
The basic idea in the proof is to transfer from horizontal distributions to vertical ones.

The lower-bound part in Theorem \ref{thm:main} starts with the restriction of moduli to some structured set $\cN$, where for each
$n=qrst\in\cN,$ $r,s,t$ are distinct primes with suitable sizes, and $q$ is a very small squarefree number. More precisely, we may assume $r,s,t$ are reasonably large so that $n$ is also squarefree. Note that the Kloosterman sum $\kl(1,qrst)$ factorizes as a product of 
$\kl(\overline{(rst)^2},q)$, $\kl(\overline{(qst)^2},r)$, 
$\kl(\overline{(qrt)^2},s)$ and $\kl(\overline{(qrs)^2},t)$.
Suppose that we always have $s<t<r$ and $t^2>r$ in $\cN$, and Lemma \ref{lm:bilinear-symmetricpower} (a very special case of Fouvry, Kowalski and Michel \cite{FKM14}) guarantees that each of the last three Kloosterman sums has equidistributions as $n$ runs over $\cN.$ By an inclusion-exclusion argument, we reduce the problem to evaluate the lower bound for
\begin{align*}
\sum_{qrst\in\cN}|\kl(\overline{(rst)^2},q)|.
\end{align*}
Note that $q$ runs over small squarefree numbers, and it essentially guarantees the density of moduli not being dropped. Such arguments did appear in numerous early works on large values and sign changes of Kloosterman sums \cite{Mi95,FM07,SF09}.
A Barban--Davenport--Halberstam type theorem of the author \cite{Xi20} is then used to evaluate the above average asymptotically.

The story of the upper-bound part in Theorem \ref{thm:main} is a bit different. In previous works \cite{FM07,SF09,Ma11,Xi15,Xi18,Xi20}, one applies the sieve method to the Kloosterman sum $|\kl(1,n)|$ to capture almost primes $n$.
On average over $n=n_1n_2$ with $(n_1,n_2)=1$ we factorize $\kl(1,n)=\kl(\overline{n_2^2},n_1)\kl(\overline{n_1^2},n_2)$ and this allows us to move from horizontal distributions to vertical ones.
At some point we apply Weil's bound for one Kloosterman sum and vertical Sato--Tate distribution for the other. However, the application of Weil's bound makes us to increase the density a bit, which is not desired.
On the other hand, the upper bound in \cite{FM03} is rather close to the matching one by a loss of $\log\log$ factor. The arguments therein also use Weil's bound
at some point, but the loss, in the author's opinion, comes from a crude factorization of the modulus. It seems also hopeful to reach the correct order of magnitude by a very careful decomposition of the modulus. 

We proceed in a somewhat different way in this paper. Instead of sifting Kloosterman sums directly, we regard the normalized Kloosterman sum $2^{-\omega(n)}|\kl(1,n)|$ as a product of local densities. We then introduce a density version of sieves in the next section (Theorem \ref{thm:sieve-densityversion}), which enables us to keep all the essential information on densities coming from Kloosterman sums.

\subsection*{Acknowledgements}
The author thanks \'Etienne Fouvry, Kaisa Matom\"aki, Philippe Michel for helpful comments and encouragements.
Part of this work was completed during the Program ``Fouvry-73'' held in the Bernoulli Centre of EPF Lausanne during the summer of 2026. We thank the organizers of this four-week program, and the hospitality and wonderful working environment are greatly appreciated. 
This work is supported in part by NSFC (No. 12625101), Shaanxi NSF (No. 2025JC-QYCX-002) and Shaanxi Fundamental Science Research Project for Mathematics and Physics (No.25JSZ007).

\subsection*{AI tools}
The author had an early idea to remove the $\log\log$-factor in \eqref{eq:FM-upper}, but no ideas on the improvement over the lower bound of Fouvry and Michel. After sending this problem to ChatGPT, an improvement was created several hours later, but reaching the correct of magnitude. This result was not quite satisfactory but the author was greatly encouraged. The idea of proving the current lower bound came to the brain of the author some days ago, without any 
omen.
ChatGPT was later used to polish the presentation and check the errors, and the author is responsible for any mathematical issues in this paper.

\smallskip

\section{A density version of the sieve method}

The upper-bound part of Theorem \ref{thm:main} is proven by means of sieve methods.
Recall that the origin of sieves was designed to capture primes in reasonable sets of integers, or more generally is to evaluate the averages of non-negative sequences, say $(a_n)$, with $n$ supported on primes. To do so, one can rule out the integers $n$ with small prime factors by using the M\"obius formula
\begin{align}\label{eq:1*mu}
(1*\mu)(n)=\sum_{d\mid n}\mu(d)=
\begin{cases}
1,\ \ &n=1,\\
0, &n>1.
\end{cases}
\end{align}
The upper and lower bound sieves $\blambda^\pm=(\lambda_d^\pm)$ are constructed under the general principle
\begin{align}\label{eq:lowerupperboundsieve}
1*\blambda^- \leqslant 1*\mu\leqslant 1*\blambda^+.
\end{align}
Generally speaking, one needs to understand the distribution of the subsequence $(a_n)_{n\equiv0\bmod d}$ with suitably large $d$. Unfortunately, one has little such information when $a_n$ is taken to be Kloosterman sums $|\kl(1,n)|,$ which prevents us using 
sieve methods in a classical setting to study $|\kl(1,n)|$ directly.
This motivates us to introduce a {\it density version} of sieves. Roughly speaking, we do not sifting the Kloosterman sums as the target sequence. Instead, we regard the normalized Kloosterman sum $2^{-\omega(n)}|\kl(1,n)|$ as a product of {\it densities}.

Let $\fP$ be a finite set of certain primes. Suppose the sequence $\bdelta=(\delta_p)_{p\in\fP}$ satisfies
\begin{align*}
0\leqslant \delta_p\leqslant 1
\end{align*}
for each $p\in\fP$. The following theorem gives upper and lower bounds for the product
\begin{align*}
\prod_{p\in\fP}\delta_p
\end{align*}
by means of sieve methods.

\begin{theorem}\label{thm:sieve-densityversion}
Let $(\lambda_d^+)$ be an upper-bound sieve, and $(\lambda_d^-)$ a lower-bound sieve. Then
\begin{align*}
\sum_{d\mid\Pi}\lambda^-_d\prod_{p\mid d}(1-\delta_p)\leqslant
\prod_{p\in\fP}\delta_p
\leqslant \sum_{d\mid\Pi}\lambda^+_d\prod_{p\mid d}(1-\delta_p),
\end{align*}
where $\Pi$ denotes the product of primes in $\fP$.
\end{theorem}

\proof
We first introduce independent Bernoulli random variables $x_p$ indexed by  $p\in\fP$. Each $x_p$ takes value $1$ with probability $\delta_p$, and takes value $0$ with probability $1-\delta_p$. This allows us to define a random set $\fX=\{p\in\fP:x_p=0\}$. Clearly,
\begin{align*}
\mathbf{1}_{\fX=\emptyset}=\sum_{d\mid\Xi}\mu(d)=\sum_{d\mid \Pi}\mu(d)\mathbf{1}_{d\mid \Xi}1,
\end{align*}
where $\Xi$ denotes the product of primes in $\fX$.

Suppose $(\lambda^+_d)$ is an upper bound sieve (of certain level and dimension). Then
\begin{align*}
\mathbf{1}_{\fX=\emptyset}
\leqslant \sum_{d\mid \Pi}\lambda^+_d\mathbf{1}_{d\mid \Xi}1.
\end{align*}
Taking expectations on both sides, we find
\begin{align*}
\prod_{p\in\fP}\delta_p
\leqslant \sum_{d\mid\Pi}\lambda^+_d\prod_{p\mid d}(1-\delta_p).
\end{align*}
This proves the upper-bound part, and the lower-bound parts follows similarly.
\endproof

We believe that Theorem \ref{thm:sieve-densityversion} will find many other applications in arithmetics. The above proof is based on the probabilistic method, from which one may see easily why we regard this as a density version of sieves. In practice, $\delta_p$ can be some fractions or even irrational numbers on average. In the appendix, we will give a completely deterministic proof.

To conclude this section, we recall the basic notation on the level and dimension in classical sieve theory.
Give a sequence $\cA=(a_n)$ of non-negative numbers and $\fP$ a set of prime numbers. Define the sifting function
\begin{align*}
S(\cA,\fP,z)=\sum_{(n,P(z))=1}a_n
\quad\text{with}\quad
P(z)=\prod_{p<z, \, p\in\fP}p.
\end{align*}
One needs to construct the upper and lower bound sieves $\blambda^\pm=(\lambda_d^\pm)$ with \eqref{eq:lowerupperboundsieve} so that
\begin{align*}
S^-(\cA,\cP,z)\leqslant S(\cA,\cP,z)\leqslant S^+(\cA,\cP,z)
\end{align*}
with
\begin{align*}
S^\pm(\cA,\fP,z):=\sum_{d\mid P(z)}\lambda_d^\pm \sum_{n\equiv 0\bmod d}a_n.
\end{align*}
For each choice of $\pm,$ we say $\lambda_d^\pm$ has {\it level} $D\geqslant1$ if $\lambda_d^\pm=0$ for all $d>D.$
For squarefree $d$ with all its prime factors belonging to $\fP$, 
we consider the subsequence $\cA_d=(a_n)_{n\equiv 0\bmod d}$.
Assume that $\cA_d$ is well located in $\cA$ in the following sense: 
There exists a multiplicative function $g$ supported on squarefree numbers with all its prime factors belonging to $\fP$ verifying
\[0<g(p)<1\quad(p\in\fP)\]
such that
\par
(a) the remainder 
\begin{align}\label{eq:sieveapproximation}
r(\cA,d)=\sum_{n\equiv 0\bmod d}a_n-g(d)\sum_{n}a_n
\end{align}
is small on average over $d\mid P(z)$;

(b) there exist some constants $\kappa\geqslant0$ and $K>1$ such that
\begin{align*}
\prod_{w\leqslant p<z}(1-g(p))^{-1}
\leqslant K\Big(\frac{\log z}{\log w}\Big)^\kappa
\end{align*}
for all $2\leqslant w< z$. The parameter $\kappa$ is called the {\it sieve dimension}.

In view of \eqref{eq:sieveapproximation}, one is led to study
\begin{align*}
\cM:=\sum_{d\mid P(z)}\lambda_d^\pm g(d),\ \ \ 
\cR:=\sum_{d\mid P(z)}\lambda_d^\pm r(\cA,d).
\end{align*}
The art of the sieve theory is to design good sieve weights $\lambda_d^\pm$ to optimize $\cM$, and one also has to choose $D$ as large as possible such that $\cR$ has a lower order of magnitude with given $\blambda^\pm=(\lambda_d^\pm)$.
The shape of $\cM$ is our original intuition to prove a density version of sieves as in Theorem \ref{thm:sieve-densityversion}.

\smallskip
\section{Basics on Kloosterman sums}

 For each $p\in\cP$ and $a\in\bZ$, the celebrated Weil's bound asserts that $|\kl(a,p)|\leqslant2,$ from which one finds there exists a unique $\theta_p(a)\in[0,\pi]$ such that 
\[\kl(a,p)=2\cos\theta_p(a).\]
The distribution of Kloosterman sums $\kl(a,p)$ can be interpreted in terms of the angles $\theta_p(a).$ Hence the above-mentioned vertical equidistribution of Katz  \cite{Ka88} can be reformulated as follows: as $p\rightarrow+\infty,$ the set $\{\theta_p(a):a\in\F_p^\times\}\subseteq[0,\pi]$ becomes equidistributed with respect to the Sato--Tate measure $\frac{2}{\pi}\sin^2\theta\ud\theta$. In fact, Katz  \cite[Page 241]{Ka88} proved an upper bound for the Weyl sum:
\begin{align*}
\Big|\sum_{a\in\F_p^\times}\sym_k(\theta_p(a))\Big|\leqslant \frac{k+1}{2}\sqrt{p}
\end{align*}
for all $k\geqslant1.$ Here $\sym_k$ denotes the $k$-th symmetric power given by
\begin{align*}
\sym_k\theta=\frac{\sin((k+1)\theta)}{\sin\theta},\ \ \theta\in[0,\pi].
\end{align*}
Alternatively, $\sym_k$ can be translated to Chebyshev polynomials of the second kind: For $x\in[-1,1]$ and $k\geqslant0,$ define the polynomials $\{U_k(x)\}_{k\geqslant0}$ recursively by
\[U_0(x)=1,\ \ U_1(x)=2x,\]
\[U_{k+1}(x)=2xU_k(x)-U_{k-1}(x).\]
It is easy to show that
\begin{align*}
\sym_k\theta=U_k(\cos\theta).
\end{align*}
The sequence of $\{U_k(x)\}_{k\geqslant0}$ then forms an orthonormal basis of $L^2([-1,1])$ with respect to the (Sato--Tate) measure $\frac{2}{\pi}\sqrt{1-x^2}\ud x.$ Hence the distributions of Kloosterman sums with various indices can be formulated with non-trivial bounds for sums of symmetric powers or Chebyshev polynomials.

The following bilinear form estimation can be found in \cite[Corollaire 2.11]{Mi95} and a more general statement has been proved in \cite[Theorem 1.17]{FKM14}.

\begin{lemma}\label{lm:bilinear-symmetricpower}
Let $p$ be a large prime and $(a,p)=1.$
For each $k\geqslant1$ and any coefficients $\boldsymbol\alpha=(\alpha_m),\boldsymbol\beta=(\beta_n),$ we have
\begin{align*}
\mathop{\sum_{m\leqslant M}\sum_{n\leqslant N}}_{(mn,p)=1}
\alpha_m\beta_n\sym_k(\theta_p(a\overline{(mn)^2}))\ll 
\|\boldsymbol\alpha\|\|\boldsymbol\beta\|(MN)^{\frac{1}{2}}(p^{-\frac{1}{4}}+N^{-\frac{1}{2}}+M^{-\frac{1}{2}}p^{\frac{1}{4}}(\log p)^{\frac{1}{2}}),\end{align*}
where the implied constant depends polynomially on $k$.
\end{lemma}

For $(a,n)=1,$ we put
\begin{align}\label{eq:Omega}
\varOmega(a,n)=\kl(\overline{a}^2,n).
\end{align}

The following lemma is a consequence of vertical equidistribution of Kloosterman sums, in a very explicit shape due to Fouvry and Michel \cite[Lemme 2.4]{FM03}.
\begin{lemma}\label{lm:Omega-firstmoment}
For all large primes $p$ with $p\nmid b,$ we have
\begin{align*}
\sum_{a\in\F_p^\times}|\varOmega(a,p)|=\frac{8}{3\pi}p+O(p^{\frac{3}{4}}).\end{align*}
and
\begin{align*}
\sum_{a\in\F_p^\times}|\cos(\theta_p(b\overline{a^2}))|=\frac{4}{3\pi}p+O(p^{\frac{3}{4}}).\end{align*}
\end{lemma}

We will also need to examine the averages of $|\varOmega(a,n)|$ attached with some weights in $a$ satisfying the  following ``Siegel--Walfisz" condition.

\begin{definition} \label{def:SW}
An arithmetic function $\vartheta$ is said to satisfy the ``Siegel--Walfisz" condition, if for any $w\geqslant1,d\geqslant1,(w,a)=1,a\neq0,$
\begin{align}\label{eq:SiegelWalfisz}
\sum_{\substack{n\leqslant x\\ n\equiv a\bmod w\\(n,d)=1}}\vartheta_n-\frac{1}{\varphi(w)}\sum_{\substack{n\leqslant x\\(n,dw)=1}}\vartheta_n\ll\|\boldsymbol\vartheta\|x^{\frac{1}{2}}\tau(d)^{B}(\log x)^{-A} \end{align}
holds for some constant $B>0$ and any $A>0$ with an implied constant in $\ll$ depending only on $A$.
\end{definition}
We are now ready to state a Barban--Davenport--Halberstam theorem for Kloosterman sums. This is a special case of Xi \cite[Lemma 4.1]{Xi20}.

\begin{lemma}\label{lm:bilinear-BDH}
Let $M,N>0$ and $q\geqslant1$ squarefree. Let $\boldsymbol\alpha=(\alpha_m)$ be a complex coefficient with support in $[M,2M]$ and also satisfy the  above``Siegel--Walfisz" condition, and $\boldsymbol\beta=(\beta_n)$ an arbitrary
complex coefficient with support in $[N,2N].$ Put
\begin{align*}
\cE(\boldsymbol\alpha,\boldsymbol\beta;q)&=\mathop{\sum\sum}_{(mn,q)=1}\alpha_m\beta_n
|\varOmega(mn,q)|-\frac{1}{\varphi(q)}\sideset{}{^*}\sum_{a\bmod q}|\varOmega(a,q)|\mathop{\sum\sum}_{(mn,q)=1}\alpha_m\beta_n\\
&=\sideset{}{^*}\sum_{a\bmod q}|\varOmega(a,q)|\Big(\mathop{\sum\sum}_{mn\equiv a\bmod q}\alpha_m\beta_n-\frac{1}{\varphi(q)}\mathop{\sum\sum}_{(mn,q)=1}\alpha_m\beta_n\Big).\end{align*}

Let $\kappa\geqslant1$ and $M\geqslant N$. For any $A>0$, there exists some constant $B=B(A)>0,$ such that
\begin{align*}
\sum_{q\leqslant Q}\mu^2(q)\tau(q)^\kappa |\cE(\boldsymbol\alpha,\boldsymbol\beta;q)|\ll \|\boldsymbol\alpha\| \|\boldsymbol\beta\|Q(MN)^{\frac{1}{2}}(\log MN)^{-A}\end{align*}
for $Q\leqslant MN(\log MN)^{-B},$
where the implied constant depends only on $A$ and $\kappa$.
\end{lemma}

\smallskip

\section{Lower bound}

To derive the lower-bound part of Theorem \ref{thm:main}, we restrict the modulus $n$ to the structured set $\cN(X),$ which consists of squarefree integers $n=qrst\leqslant X$ with $r,s,t$ primes and
\begin{align}\label{eq:qrst-box}
q\leqslant X^{0.001},~~
r>X^{0.44},~~
X^{0.2}<s\leqslant X^{0.22},~~
X^{0.3}<t\leqslant X^{0.33}.
\end{align}
Automatically we have $q$ is squarefree,
\begin{align*}
X^{0.44}<r\leqslant X^{0.5},
\end{align*}
and $q,r,s,t$ are pairwise coprime.
Note that over the box given by \eqref{eq:qrst-box},
the map $(q,r,s,t)\mapsto qrst$ is injective, which is obvious due to the sizes of $r,s,t$.
It is also useful to observe that
\begin{align*}
|\cN(X)|
&=\sum_{q\leqslant X^{0.001}}\sum_{X^{0.2}<s\leqslant X^{0.22}}\sum_{X^{0.3}<t\leqslant X^{0.33}}\sum_{X^{0.44}<r\leqslant X/(qst)}1\\
&\gg \sum_{q\leqslant X^{0.001}}\sum_{X^{0.2}<s\leqslant X^{0.22}}\sum_{X^{0.3}<t\leqslant X^{0.33}}\frac{X}{qst\log (X/qst)}\\
&\gg X.
\end{align*}
Hence we are choosing a positive proportion  of integers in $[1,X]$ with designed structures.
We then obtain an lower bound
\begin{align*}
\fS(X)\geqslant \sum_{qrst\in\cN(X) }|\kl(1,qrst)|=\sum_{qrst\in\cN(X) }|\varOmega(rst,q)\varOmega(qst,r)\varOmega(qrt,s)\varOmega(qrs,t)|.
\end{align*}
Denote by $\cN_1(X)$ the set which consists of squarefree numbers $n=qrst\in\cN(X)$ such that 
\begin{align*}
\min\{|\varOmega(qst,r)|,|\varOmega(qrt,s)|,|\varOmega(qrs,t)|\}>2\eta,
\end{align*}
where $\eta=10^{-2026}.$ Put
\begin{align*}
\fS_1(X):=\sum_{qrst\in\cN_1(X) }|\varOmega(rst,q)|.
\end{align*}

The following lemma is obvious.
\begin{lemma}\label{lm:S-S1}
For all large $X,$ we have
\begin{align*}
\fS(X)>8\eta^3 \fS_1(X).
\end{align*}
\end{lemma}

The core is thus to prove the lower bound $\fS_1(X)\gg X(\log X)^{\frac{8}{3\pi}-1}.$ A trivial upper bound for $\fS_1(X)$ is
\begin{align*}
\fS_1(X)\leqslant \fS_0(X)
\end{align*}
with
\begin{align*}
\fS_0(X):=\sum_{qrst\in\cN(X) }|\varOmega(rst,q)|.
\end{align*}
The following lemma gives an inequality in an opposite direction.
\begin{proposition}\label{prop:S1-S0}
For all large $X,$ we have
\begin{align*}
\fS_1(X)\gg  \fS_0(X).
\end{align*}
\end{proposition}

\begin{proposition}\label{prop:S0-asymptotic}
For all large $X,$ we have
\begin{align*}
\fS_0(X)\asymp X(\log X)^{\frac{8}{3\pi}-1}.
\end{align*}
\end{proposition}

We need the inclusion-exclusion principle in the following shape.

\begin{lemma}\label{lm:inclusion-exclusion}
Let $k$ be a fixed positive integer.  Let $\cA$ be a set equipped with a finite measure $\mu$ and  $\cA_1,\cdots,\cA_k$ as measurable subsets of $\cA$ . Then
\begin{align*}
 \mu(\cA_1\cap\cdots\cap\cA_k)
 \ge \mu(\cA_1)+\cdots+\mu(\cA_k)-(k-1)\mu(\cA).
\end{align*}
\end{lemma}

\proof[Proof of Proposition \ref{prop:S1-S0}]
We apply Lemma \ref{lm:inclusion-exclusion} with $k=3,$
\begin{align*}
\cA_1&=\{qrst\in\cN(X): |\varOmega(qst,r)|>2\eta\},\\
\cA_2&=\{qrst\in\cN(X): |\varOmega(qrt,s)|>2\eta\},\\
\cA_3&=\{qrst\in\cN(X): |\varOmega(qrs,t)|>2\eta\}
\end{align*}
and $\cA=\cN(X).$ The measure $\mu$ therein is taken to be
\begin{align*}
\mu(\cB)=\sum_{qrst\in \cB}|\varOmega(rst,q)|,\ \ \cB\subseteq \cN(X).
\end{align*}
We claim that, for $j=1,2,3,$
\begin{align}\label{eq:mu(Aj)-lowerbound}
\mu(\cA_j)
&>0.999\fS_0(X)+O(X(\log X)^{-2}).
\end{align}
Hence by Lemma \ref{lm:inclusion-exclusion}, we derive that
\begin{align*}
\fS_1(X)&=\mu(\cA_1\cap\cA_2\cap\cA_3)
\geqslant \mu(\cA_1)+\mu(\cA_2)+\mu(\cA_3)-2\mu(\cA) \\
&>0.996\fS_0(X)+O(X(\log X)^{-2}),
\end{align*}
as desired.

We now turn to prove \eqref{eq:mu(Aj)-lowerbound} for $j=1,$ and the other cases follow similarly. In fact,
\begin{align*}
\mu(\cA_1)
&=\sum_{qrst\in \cA_1}|\varOmega(rst,q)|
=\sum_{\substack{qrst\in \cN(X)\\ |\varOmega(qst,r)|>2\eta}} |\varOmega(rst,q)|=\sum_{qrst\in \cN(X)} |\varOmega(rst,q)|\delta_\eta\Big(\frac{\varOmega(qst,r)}{2}\Big)
\end{align*}
where $\delta_\eta(x)=\mathbf{1}_{|x|>\eta}$ is defined over $x\in[-1,1]$.
Recall the Erd\H{o}s–Tur\'an type inequality\footnote{This is contained implicitly in \cite[Section 2]{FM03}, and one also use \cite[Lemma C.1]{Xi20} to give a proof.}
\begin{align*}
\Big|\sum_{r\leqslant R}\delta_\eta(x_r)-\delta_\eta \cdot R\Big|
\ll \frac{R}{K}+\sum_{1 \leqslant k \leqslant K} \frac{1}{k}\Big|\sum_{r\leqslant R} U_k(x_r)\Big|
\end{align*}
for any sequence $\{x_r\}_{r\leqslant R}\subseteq[-1,1],$ where $K>1$ is arbitrary, and
\begin{align*}
\delta_\eta&=\frac{2}{\pi}\int_{\eta<|t|\leqslant 1}\sqrt{1-t^2}\ud t.
\end{align*}
Therefore,
\begin{align}\label{eq:mu(A1)}
\mu(\cA_1)
&=(\delta_\eta+O(K^{-1}))\fS_0(X)+\sum_{1\leqslant k\leqslant K}\alpha_k W_k,
\end{align}
where $\alpha_k\ll1/k$ and
\begin{align*}
W_k:=\sum_{qrst\in \cN(X)} |\varOmega(rst,q)|U_k\Big(\frac{\varOmega(qst,r)}{2}\Big)
=\sum_{qrst\in \cN(X)} |\varOmega(rst,q)|\sym_k(\theta_r(\overline{(qst)^2})).
\end{align*}

For each Dirichlet character $\chi\bmod q$, define the Mellin transform
\begin{align*}
\widetilde{\varOmega}(\chi,q):=\frac{1}{\sqrt{\varphi(q)}}~\sideset{}{^*}\sum_{a\bmod q}\overline{\chi}(a)|\varOmega(a,q)|.
\end{align*}
By orthogonality, we have the inversion
\begin{align*}
|\varOmega(rst,q)|=\frac{1}{\sqrt{\varphi(q)}}\sum_{\chi\bmod q}\chi(rst)\widetilde{\varOmega}(\chi,q),
\end{align*}
so that we may write
\begin{align*}
W_k
&=\sum_{q\leqslant X^{0.001}} \frac{1}{\sqrt{\varphi(q)}}\sum_{\chi\bmod q}\widetilde{\varOmega}(\chi,q)\sum_{\substack{ r,s,t\text{ primes} \\ qrst\in \cN(X)}} \chi(r)\chi(s)\chi(t)\sym_k(\theta_r(\overline{(qst)^2})).
\end{align*}
To bound $W_k$ from above, we regard the inner sums over $s,t$ as a bilinear form with symmetric $k$-th power of Kloosterman sums. In view of Lemma \ref{lm:bilinear-symmetricpower}, one may get power-savings against the trivial bound. Hence
\begin{align*}
W_k
&\ll k^A X^{-0.025}\sum_{q\leqslant X^{0.001}} \frac{1}{\sqrt{\varphi(q)}}\sum_{\chi\bmod q}|\widetilde{\varOmega}(\chi,q)|\sum_{\substack{ r,s,t\text{ primes} \\ qrst\in \cN(X)}} 1
\end{align*}
for some absolute constant $A>0$.
By Cauchy inequality we obtain
\begin{align*}
W_k
&\ll k^A X^{-0.02}|\cN(X)|.
\end{align*}
Summing over $k\leqslant K$, we find
\begin{align*}
\sum_{1\leqslant k\leqslant K}\alpha_k W_k
&\ll K^{A+1} X^{-0.02}|\cN(X)|.
\end{align*}

Note that
\begin{align*}
\delta_\eta=\frac{4}{\pi}\int_{\eta}^1\sqrt{1-t^2}\ud t
&=1-\frac{2}{\pi}(\eta\sqrt{1-\eta^2}+\arcsin \eta).
\end{align*}
The observation $\arcsin \eta<2\eta$ yields $\delta_\eta>1-\frac{2}{\pi}(\eta\sqrt{1-\eta^2}+2 \eta)>1-\frac{6}{\pi}\eta>0.9999.$
Recalling \eqref{eq:mu(A1)}, collecting the above arguments and choosing $K$ as a sufficiently small power of $X$, we then conclude \eqref{eq:mu(Aj)-lowerbound} for $j=1$, and thus for $j=2,3$ similarly.
We now complete the proof of this proposition.
\endproof

\proof[Proof of Proposition \ref{prop:S0-asymptotic}]

The indicator function of prime numbers, as well as the convolution with itself, satisfies the ``Siegel--Walfisz" condition as defined in Definition \ref{def:SW}. 
From Lemma \ref{lm:bilinear-BDH} it follows that
\begin{align*}
\fS_0(X)
&=\sum_{qrst\in \cN(X)}\Big(\frac{8}{3\pi}\Big)^{\omega(q)}W(q)+O(X(\log X)^{-2}),
\end{align*}
where $W$ is a multiplicative function supported over squarefree numbers, and for each $p\in\cP$,
\begin{align*}
W(p)=\frac{3\pi}{8}\frac{1}{\varphi(p)}\sideset{}{^*}\sum_{a\bmod p}|\varOmega(a,p)|.
\end{align*}
Lemma \ref{lm:Omega-firstmoment} yields
\begin{align*}
W(p)=1+O(p^{-\frac{1}{4}}).
\end{align*}
Therefore, the prime number theorem yields
\begin{align*}
\sum_{qrst\in \cN(X)}\Big(\frac{8}{3\pi}\Big)^{\omega(q)}W(q)
&=\sum_{q\leqslant X^{0.001}}\Big(\frac{8}{3\pi}\Big)^{\omega(q)}W(q)\sum_{X^{0.2}<s\leqslant X^{0.22}}\sum_{X^{0.3}<t\leqslant X^{0.33}}\sum_{r\leqslant X/(qst)}1\\
&\asymp \frac{X}{\log X}\sum_{q\leqslant X^{0.001}}\Big(\frac{8}{3\pi}\Big)^{\omega(q)}\frac{W(q)}{q}\sum_{X^{0.2}<s\leqslant X^{0.22}}\frac{1}{s}\sum_{X^{0.3}<t\leqslant X^{0.33}}\frac{1}{t}\\
&\asymp \frac{X}{\log X}\sum_{q\leqslant X^{0.001}}\Big(\frac{8}{3\pi}\Big)^{\omega(q)}\frac{W(q)}{q}.
\end{align*}
A routine application of Perron's formula and Selberg--Delange method shows that the last sum over $q$ is $(\log X)^{\frac{8}{3\pi}}.$ Now Proposition \ref{prop:S0-asymptotic} follows immediately.

A direct combination of Propositions \ref{prop:S1-S0} and \ref{prop:S0-asymptotic} and Lemma \ref{lm:S-S1} proves the lower-bound part of Theorem \ref{thm:main}.
\endproof

\smallskip

\section{Upper bound}

We first work with averages of Kloosterman sums with squarefree moduli:
\begin{align*}
\fS^\flat (X;a,b):=\sum_{\substack{n\leqslant X\\ (n,ab)=1}}\mu^2(n)|\kl(a\overline{b},n)|,
\end{align*}
where $a,b$ are non-zero integers. We prove the following upper bound for $\fS^\flat (X;a,b)$ with the aid of Theorem \ref{thm:sieve-densityversion}.

\begin{proposition}\label{prop:Sflat-upperbound}
For all large $X$ and integers $a,b$ with $1\leqslant a,b\leqslant X/2,$ we have 
\begin{align*}
\fS^\flat (X;\pm a,b)
&\ll \Big(\frac{\varphi(ab)}{ab}\Big)^{\frac{8}{3\pi}} X(\log X)^{\frac{8}{3\pi}-1} .
\end{align*}
\end{proposition}

To begin with,  we first evaluate mean values of $2^{\omega(n)}$ in arithmetic progressions.

\begin{lemma}\label{lm:divisorinAPs}
For all $q,t\in\Z^+$ with $(q,ar)=1,$ we have
\begin{align*}
\sum_{\substack{n\leqslant x\\ n\equiv a\bmod q\\ (n,r)=1}}\mu^2(n)2^{\omega(n)}
&=\frac{C_0C_1(qr)}{\varphi(q)}x(\log x+B+C_2(qr))+O(q^{-\frac{1}{4}} x^{\frac{1}{2}}(qrx)^{\varepsilon})
\end{align*}
for any $\varepsilon>0,$ where the implied $O$-constant depends at most on $\varepsilon$ and is uniform for $(a,q)=1.$ Here
\begin{align*}
C_0=\prod_p\left(1-\frac1p\right)^2\left(1+\frac2p\right),\ \ B=2\gamma-1+6\sum_p\frac{\log p}{(p-1)(p+2)},\end{align*}
\begin{align*}C_1(m)&=\prod_{p\mid m}\frac{p}{p+2},\ \
C_2(m)=2\sum_{p\mid m}\frac{\log p}{p+2},
\end{align*}
where $\gamma\approx 0.57721$ denotes Euler's constant.

\end{lemma}

\proof
The statement is essentially due to Selberg \cite{Se91} and Hooley \cite{Ho57}, who studied the averages with $r=1$, without the weight $\mu^2(n)$ or with the divisor function $\tau(n)$ in place of $2^{\omega(n)}$.
The proof relies on Poisson summation and Weil's bound for individual Kloosterman sums.
In fact, we can show that
\begin{align*}
\sum_{\substack{n\leqslant x\\ n\equiv a\bmod q\\ (n,r)=1}}\mu^2(n)2^{\omega(n)}
&=\frac{1}{\varphi(q)}\sum_{\substack{n\leqslant x\\ (n,qr)=1}}\mu^2(n)2^{\omega(n)}+O(q^{-\frac{1}{4}} x^{\frac{1}{2}+\varepsilon}),
\end{align*}
where the implied constant depends at most on $\varepsilon$ and is uniform for $(a,q)=1.$
By a routine application of Perron's formula, we can show that
\begin{align*}
\sum_{\substack{n\leqslant x\\ (n,qr)=1}}\mu^2(n)2^{\omega(n)}
&=C_0C_1(qr)x(\log x+B+C_2(qr))+O(x^{\frac{1}{2}+\varepsilon})
\end{align*}
for any $\varepsilon>0.$
This completes the proof of the lemma.
\endproof

\proof[Proof of Proposition \ref{prop:Sflat-upperbound}]
For each squarefree number $n$ with $(n,ab)=1$ and $z>2$, we have
\begin{align*}
|\kl(a\overline{b},n)|\leqslant 2^{\omega(n)}\prod_{\substack{p\mid n\\ p<z}}|\cos(\theta_p(a\overline{b(n/p)^2}))|.
\end{align*}
Taking $\delta_p=|\cos(\theta_p(a\overline{b(n/p)^2}))|$ in Theorem \ref{thm:sieve-densityversion}, we find
\begin{align*}
|\kl(a\overline{b},n)|\leqslant 2^{\omega(n)}\sum_{d\mid(n,P_{ab}(z))}\lambda^+_d\prod_{p\mid d}(1-|\cos(\theta_p(a\overline{b(n/p)^2}))|)
\end{align*}
for each upper bound sieve $(\lambda^+_d)$, where $P_{ab}(z)$ denotes the product of primes $<z$ and not dividing $ab$. Note that the above inequality is uniform for all squarefree numbers $n$.
Therefore,
\begin{align*}
\fS^\flat (X;a,b)
&\leqslant \sum_{\substack{n\leqslant X\\ (n,ab)=1}}\mu^2(n)2^{\omega(n)}\sum_{d\mid(n,P_{ab}(z))}\lambda^+_d\prod_{p\mid d}(1-|\cos(\theta_p(a\overline{b(n/p)^2}))|).
\end{align*}
From now on, we assume that $(\lambda^+_d)$ is an upper bound beta sieve (see \cite[Chapter 11]{FI10} for details) of level $D$ and dimension $\kappa$ with
\begin{align*}
D=X^{0.1},\ \ \ \kappa=2-\frac{8}{3\pi}.
\end{align*}
We also fix $z=D^{1/s_0}$ with $s_0$ a sufficiently large fixed constant.
Switching summations, we obtain
\begin{align}
\fS^\flat (X;a,b)
&\leqslant\sum_{d\mid P_{ab}(z)}\lambda^+_d \sum_{\substack{n\leqslant X\\ d\mid n,~(n,ab)=1}}\mu^2(n)2^{\omega(n)}\prod_{p\mid d}(1-|\cos(\theta_p(a\overline{b(n/p)^2}))|)\nonumber\\
&=\sum_{d\mid P_{ab}(z)}\lambda^+_d 2^{\omega(d)} \sum_{\substack{n\leqslant X/d\\ (n,abd)=1}}\mu^2(n)2^{\omega(n)}\prod_{p\mid d}(1-|\cos(\theta_p(a\overline{b(nd/p)^2}))|).\label{eq:Sflat-upperboundsieve}
\end{align}

By virtue of Lemma \ref{lm:divisorinAPs}, the inner sum over $n$ in \eqref{eq:Sflat-upperboundsieve} can be transformed as
\begin{align*}
\sum_n&=\sideset{}{^*}\sum_{v\bmod d}\prod_{p\mid d}(1-|\cos(\theta_p(a\overline{b(vd/p)^2}))|)\sum_{\substack{n\leqslant X/d\\ n\equiv v\bmod d\\ (n,ab)=1}}\mu^2(n)2^{\omega(n)}\\
&=C_0C_1(ab)2^{-\omega(d)} g(d)\cdot X\{\log (X/d)+B+C_2(ab)+C_2(d)\}+O(d^{\frac{3}{4}}X^{\frac{1}{2}+\varepsilon}).
\end{align*}
where $g$ is a multiplicative function supported on squarefree numbers with
\begin{align*}
g(p)&=\frac{2C_1(p)}{p}\frac{1}{\varphi(p)}\sideset{}{^*}\sum_{v\bmod p}(1-|\cos(\theta_p(a\overline{bv^2}))|).
\end{align*}
Note that Lemma \ref{lm:Omega-firstmoment} yields
\begin{align*}
g(p)=\frac{2-\frac{8}{3\pi}}{p}(1+O(p^{-\frac{1}{4}}))=\frac{\kappa}{p}(1+O(p^{-\frac{1}{4}})).
\end{align*}

Therefore, it follows from \eqref{eq:Sflat-upperboundsieve} that
\begin{align}\label{eq:Sflat-densityaverage}
\fS^\flat (X;a,b)
&\leqslant C_1(ab) X\sum_{d\mid P_{ab}(z)}\lambda^+_d g(d)\cdot \{\log (X/d)+B+C_2(ab)+C_2(d)\}+O(X^{0.9}).
\end{align}

\begin{lemma}
With the above notation, we have
\begin{align}
\sum_{d\mid P_{ab}(z)}\lambda^+_d g(d)&\ll V_{ab}(z),\label{eq:density-average1}\\
\sum_{d\mid P_{ab}(z)}\lambda^+_d g(d)\log(X/d)&\ll V_{ab}(z)\log X,\label{eq:density-average2}\\
\sum_{d\mid P_{ab}(z)}\lambda^+_d g(d)C_2(d)&\ll V_{ab}(z),\label{eq:density-average3}
\end{align}
where
\begin{align*}
V_{ab}(z)=\prod_{p<z,p\nmid ab}(1-g(p)).
\end{align*}
\end{lemma}

\proof
The inequality \eqref{eq:density-average1} follows directly from the fundamental lemma in beta sieves; see \cite[Lemma 6.14]{FI10}.
To prove \eqref{eq:density-average2}, we write
\begin{align*}
\log (X/d)=\log (X/D)+\log (D/d).
\end{align*}
It suffices to prove that
\begin{align*}
\sum_{d\mid P_{ab}(z)}\lambda^+_d g(d)\log(D/d)&\ll V_{ab}(z)\log D.
\end{align*}
This can be examined by considering the (finite) sum
\begin{align*}
\cF(s):=\sum_{d\mid P_{ab}(z)}\lambda^+_d g(d)(D/d)^s.
\end{align*}
Note that $g$ is non-negative, and for any $\sigma>0,$
\begin{align*}
\sum_{p}g(p)p^{-\sigma}<+\infty.
\end{align*}
For $s\in\bC$ with $|s|= 1/\log D$, we also have $\cF(s)\ll V_{ab}(z)$ by the construction of beta sieve weights. From Cauchy's integration formula, we derive that
\begin{align*}
\sum_{d\mid P_{ab}(z)}\lambda^+_d g(d)\log(D/d)
&=\cF'(0)=\frac{1}{2\pi i}\oint_{|s|= 1/\log D}\frac{\cF(s)}{s^2}\ud s\\
&\ll V_{ab}(z)\oint_{|s|= 1/\log D}\frac{1}{|s|^2}|\ud s|\ll V_{ab}(z)\log D,
\end{align*}
as desired.

We now turn to \eqref{eq:density-average3}. Noticing that $d\mapsto C_2(d)$ is additive, we may consider the auxiliary (finite) sum
\begin{align*}
\cG(s):=\sum_{d\mid P_{ab}(z)}\lambda^+_d g(d)\ue^{sC_2(d)}
\end{align*}
for $s\in\bC$.
Note that for all $s,$ the function $d\mapsto \ue^{sC_2(d)}$ is multiplicative.
Also,
\begin{align*}
C_2(d)\ll\log\log (10d).
\end{align*}
For  $|s|=\eta$ with sufficiently small constant $\eta>0$, we also have $\cG(s)\ll V_{ab}(z).$ Again by Cauchy's integration formula, we have
\begin{align*}
\sum_{d\mid P_{ab}(z)}\lambda^+_d g(d)C_2(d)
&=\cG'(0)=\frac{1}{2\pi i}\oint_{|s|=\eta}\frac{\cG(s)}{s^2}\ud s\ll V_{ab}(z).
\end{align*}
We are done!
\endproof

Back to \eqref{eq:Sflat-densityaverage} we have
\begin{align}\label{eq:Sflat-partialproduct}
\fS^\flat (X;a,b)
&\ll C_1(ab) V_{ab}(z) X\log X+X^{0.9}.
\end{align}
From Mertens' formula, it follows that
\begin{align*}
V_{ab}(z)\asymp \prod_{p<X,p\nmid ab}\Big(1-\frac{1}{p}\Big)^\kappa
\asymp (\log X)^{-\kappa} \prod_{p\mid ab}\Big(1-\frac{1}{p}\Big)^{-\kappa}
=(\log X)^{-\kappa} \Big(\frac{\varphi(ab)}{ab}\Big)^{-\kappa}.
\end{align*}
Also,
\begin{align*}
C_1(ab)
&=\Big(\frac{\varphi(ab)}{ab}\Big)^2\cdot \prod_{p\mid ab}\Big(1-\frac{1}{p}\Big)^{-2}\Big(1+\frac{2}{p}\Big)^{-1}\asymp\Big(\frac{\varphi(ab)}{ab}\Big)^2.
\end{align*}
Hence
\begin{align*}
C_1(ab) V_{ab}(z)
&\ll (\log X)^{-\kappa} \Big(\frac{\varphi(ab)}{ab}\Big)^{2-\kappa}
=(\log X)^{\frac{8}{3\pi}-2} \Big(\frac{\varphi(ab)}{ab}\Big)^{\frac{8}{3\pi}},
\end{align*}
from which and \eqref{eq:Sflat-partialproduct} we conclude Proposition \ref{prop:Sflat-upperbound}.
\endproof

We are now able to complete the proof of the upper-bound part of Theorem \ref{thm:main}.
Each integer $n>1$ can be expressed uniquely as $n=n^\flat n^\sharp,$ where $n^\flat ,$ $n^\sharp$ denote the squarefree and squarefull parts of $n$, respectively.
That is,
\begin{align*}
n^\flat =\prod_{p\parallel n}p,\ \ \ n^\sharp=\prod_{p^\nu\parallel n,\nu\geqslant2}p^\nu.
\end{align*}
Automatically we have $(n^\flat,n^\sharp)=1.$ Hence
\begin{align*}
|\kl(1,n)|=|\kl(\overline{n^\sharp}^2,n^\flat)\kl(\overline{n^\flat}^2,n^\sharp)|
\leqslant |\kl(\overline{n^\sharp}^2,n^\flat)|\tau(n^\sharp).
\end{align*}
We now derive that
\begin{align*}
\fS(X)
&\leqslant  \sum_{\substack{n\leqslant X\\ n^\sharp \leqslant X^{0.2}}} |\kl(\overline{n^\sharp}^2,n^\flat)|\tau(n^\sharp)+\sum_{\substack{n\leqslant X\\ n^\sharp > X^{0.2}}}\tau(n)=\fS^{\leqslant}(X)+\fS^>(X),
\end{align*}
say.
On one hand,
\begin{align*}
\fS^>(X)
&\leqslant \sum_{\substack{m> X^{0.2}\\ m\text{ squarefull}}} \tau(m)\sum_{n\leqslant X/m}\tau(n)
\ll X\log X\sum_{\substack{m> X^{0.2}\\ m\text{ squarefull}}}\frac{\tau(m)}{m}\ll X^{0.99}.
\end{align*}
On the other hand,
\begin{align*}
\fS^{\leqslant}(X)
&=\sum_{\substack{m\leqslant X^{0.2}\\ m\text{ squarefull}}} \tau(m)\sum_{\substack{X^{0.7}<n\leqslant X/m\\ (n,m)=1}}\mu^2(n) |\kl(\overline{m}^2,n)|+O(X^{0.99}).
\end{align*}
Using Proposition \ref{prop:Sflat-upperbound} to estimate the inner sum over $n,$ it follows that
\begin{align*}
\fS^{\leqslant}(X)
&\ll X(\log X)^{\frac{8}{3\pi}-1} \sum_{\substack{m\leqslant X^{0.2}\\ m\text{ squarefull}}} \frac{\tau(m)}{m}\Big(\frac{\varphi(m)}{m}\Big)^{\frac{8}{3\pi}}+O(X^{0.99})\\
&\ll X(\log X)^{\frac{8}{3\pi}-1} .
\end{align*}
Combining the above estimates proves the upper-bound part of Theorem \ref{thm:main}.

\smallskip

\section{Comments and visions}
\label{sec:comments}

\subsection{Possible extensions}
Note that the arguments in this paper also allow further extensions of Theorem \ref{thm:general}. For instance, one may restrict $n$ to the residue class $n\equiv n_0\bmod Q$ with $Q<X^{1-\varepsilon}.$ 
We also believe the density version of sieves (Theorem \ref{thm:sieve-densityversion}) will find other applications in distributions of irregular sequences, say algebraic exponential sums indexed by the moduli.
We are able to mention some of the instances here:
\begin{enumerate}
\item For non-zero integers $a_1,a_2,\cdots,a_k$ and polynomials $f_1,f_2,\cdots,f_k\in\Z[T]$, consider the average
\begin{align*}
\sum_{n\leqslant X}\prod_{1\leqslant j\leqslant k}|\kl(a_j,f_j(n))|.
\end{align*}
One can also restrict the summation to $(f(n_j),a_j)=1$ for each $1\leqslant j\leqslant k,$ and $\kl$ can also be replaced by hyper-Kloosterman sums of higher ranks (see \cite{De80} for instance). Here we make the convention that $\kl(1,n)=\kl(1,|n|)$ for negative $n$ and $\kl(a,0)=0.$ As a typical example, one can prove, using Theorem \ref{thm:sieve-densityversion}, that
\begin{align*}
\sum_{n\leqslant X}|\kl(1,n^2+1)|\ll X(\log X)^{\frac{8}{3\pi}-1}.
\end{align*}
Note that the factorization types of $f_j$ are quite crucial. For instance, we have $\kl(1,8(2n+1))=0$ for any $n\in\Z^+$, and 
\begin{align*}
\sum_{n\leqslant X}|\kl(1,n^2)|\gg X(\log X)^{\frac{4}{\pi}}.
\end{align*}

\item Let $f,g,h\in\Z[T]$ be polynomials of degrees at least $2$ and $a,b,c\in\Z$ with $ac\neq0$. Define
\begin{align*}
W_f(a,n)&=\frac{1}{\sqrt{n}}\sum_{u\bmod n}\ue\Big(\frac{af(u)}{n}\Big),\\
V_g(b,n)&=\frac{1}{\sqrt{n}}\sum_{u\bmod n}\ue\Big(\frac{g(u)-bu}{n}\Big),\\
R_h(b,n)&=\sum_{\substack{u\bmod n\\ h(u)\equiv0\bmod n}}\ue\Big(\frac{cu}{n}\Big).
\end{align*}
Consider the average
\begin{align*}
\sum_{n\leqslant X}|W_f(a,n)|^{\nu_1}|V_g(b,n)|^{\nu_2}|R_h(c,n)|^{\nu_3}
\end{align*}
where $\nu_1,\nu_2,\nu_3$ are fixed non-negative real numbers. These sums already appeared in numerous occasions. The first one was recently studied systematically by Kowalski and Soundararajan \cite{KS22}. The second one can be regarded as a Fourier transform of the exponential $\ue(g(\cdot )/n)$, which has been studied in depth by Katz \cite{Ka90} and a vertical Sato-Tate distribution was also proven. The last one reveals the location of polynomial roots modulo $n$, which has been studied much earlier by Hooley \cite{Ho64,Ho67}.

\item Consider the $V$-function of Kowalski--Soundararajan \cite{KS22} which satisfies the twisted multiplicativity. Theorem \ref{thm:sieve-densityversion} allows one to remover the $\log\log$ factor in the relevant upper bound therein.

\end{enumerate}

\subsection{Number variance for $\SL(2,\Z)\setminus \mathbb{H}$}
A much more interesting application of the lower bound for $\fS(X)$ (in fact the second moment instead) was established by Li and Sarnak \cite{LS04}. They studied the number variance on the modular surface and integral moment of the error term in Weyl’s law. The previous approach was built by Selberg in 1975 (see \cite[Page 303]{He76}), based on his own trace formula. Alternatively, Li and Sarnak \cite{LS04} showed how to perform the Kuznetsov trace formula to associate the question with second moments of $|\kl(1,n)|$.
As a consequence, this motivates us to believe that the
variance of the number of eigenvalues of the Laplacian for the modular
surface should be Poissonian, as the interval is short enough. See also some comments in \cite{Mi06}.
We do not go into this issue in this paper, and we hope we may revisit this application in the forthcoming works.

\subsection{Prediction of the leading constant}
The existence of positive constants $C_1,C_2$ such that
\begin{align*}
C_1 X(\log X)^{\frac{8}{3\pi}-1}\leqslant \fS(X)\leqslant C_2 X(\log X)^{\frac{8}{3\pi}-1}
\end{align*}
is implicitly contained in the above proof. A more elaborate task is to reduce the admissible values of $C_2$ and $1/C_1$. 

On the other hand, we even do not know any reasonable way to conjecture the asymptotic formula for $\fS(X).$ However, an audacious guess is that
\begin{align}\label{eq:fS(X)-conjecture}
\fS(X)\sim \fS_{\mathrm{model}} X(\log X)^{\frac{8}{3\pi}-1},
\end{align}
where $\fS_{\mathrm{model}}\approx 0.87258044262552$ is defined by
\begin{align*}
\fS_{\mathrm{model}}
=\frac{1}{\Gamma(\frac{8}{3\pi})}\prod_p\Big(1-\frac{1}{p}\Big)^{\frac{8}{3\pi}}
\Big(1+\sum_{\nu \geq 1} \frac{\beta_{p, \nu}}{p^\nu}\Big) .
\end{align*}
Here for each prime power $p^\nu$, we define
\begin{align*}
\beta_{p, \nu}:=\frac{1}{\varphi(p^\nu)}~
\sideset{}{^*}\sum_{u \bmod {p^\nu}}|\kl(u^2,p^\nu)| .
\end{align*}
As shown by Lemma \ref{lm:Omega-firstmoment}, 
We have $\beta_{p,1}\rightarrow \frac{8}{3 \pi}$ as $p\rightarrow+\infty.$
The above guess is a consequence of regarding the Kloosterman factors
$\kl(\overline{p_2^{2\nu_2}},p_1^{\nu_1})$ and $\kl(\overline{p_1^{2\nu_1}},p_2^{\nu_2})$ to be independent on average over distinct primes $p_1,p_2.$ This looks unreasonable theoretically, and another observation is that small prime factors of the moduli are quite essential to the 
leading constant, as shown in the proof of Theorem \ref{thm:main}.

On the other hand, we may formally define a multiplicative function $\beta$ by
\begin{align*}
\beta(n)=\prod_{p^\nu \parallel n}\beta_{p,\nu}.
\end{align*}
and consider the mean value
\begin{align*}
B(X)=\sum_{n\leqslant X}\beta(n).
\end{align*}
Using the Selberg--Delange method, we can prove that
\begin{align}\label{eq:B(X)-asymptotic}
B(X)=\fS X(\log X)^{\frac{8}{3\pi}-1}\sum_{0\leqslant j\leqslant J}\frac{u_j}{(\log X)^j}\Big\{1+O\Big(\frac{1}{\log X}\Big)\Big\}
\end{align}
for any $J\in\Z^+$, where $\fS$ and $u_j$'s are constants that can be made explicitly. In particular, $\fS=\fS_{\mathrm{model}}$ and $u_0=1.$
More audaciously, suppose that the asymptotic expansion \eqref{eq:B(X)-asymptotic} still holds with $\fS(X)$ in place of $B(X).$
It seems meaningful to see if these give more precise and reasonable predictions.
Anyway, we are open to truly reasonable heuristics to predict the leading constant  in \eqref{eq:fS(X)-conjecture}.

\appendix

\section{A deterministic proof of Theorem \ref{thm:sieve-densityversion} }

In place of the probabilistic method in the above proof of Theorem \ref{thm:sieve-densityversion}, we now give a deterministic and algebraic proof.

We start with the obvious identity
\begin{align*}
1=\prod_{p\in\fP}(1-\delta_p+\delta_p)=\sum_{S\subseteq \fP}Z(S),
\end{align*}
where for each $S\subseteq \fP,$ we define
\begin{align*}
Z(S)=\prod_{p\in S}(1-\delta_p) \cdot \prod_{p\in\fP\setminus S}\delta_p.
\end{align*}
Clearly, 
\begin{align*}
Z(\emptyset)= \prod_{p\in\fP}\delta_p
\end{align*}
is the target we want to bound from above and below.

To introduce the sieve weight, a trivial observation is $\mathbf{1}_{S=\emptyset}=1$ if and only if $\Pi(S)=1$, where
\begin{align*}
\Pi(S)=\prod_{p\in S}p.
\end{align*}
Furthermore, we have
\begin{align*}
\mathbf{1}_{\Pi(S)=1}=\sum_{d\mid \Pi(S)}\mu(d).
\end{align*}
Therefore, for each upper bound sieve $(\lambda_d^+)$, we infer
\begin{align*}
Z(\emptyset)=\sum_{S\subseteq \fP}Z(S)\mathbf{1}_{S=\emptyset}
\leqslant \sum_{S\subseteq \fP}Z(S)\sum_{d\mid \Pi(S)}\lambda_d^+.
\end{align*}
Switching summations,
\begin{align}\label{eq:Z(emptyset)-upperbound}
Z(\emptyset)
&\leqslant \sum_{d\mid \Pi}\lambda_d^+\sum_{\substack{S\subseteq \fP\\ d\mid\Pi(S)}}Z(S).
\end{align}
The condition $S\subseteq \fP,d\mid\Pi(S)$ is equivalent to that $S$ contains all prime factors of $d$. Hence
\begin{align*}
\sum_{\substack{S\subseteq \fP\\ d\mid\Pi(S)}}Z(S)
&=\prod_{p\mid d}(1-\delta_p) \cdot\sum_{\substack{T\subseteq \fP\\ p\mid d\Rightarrow p\not\in T}}\prod_{p\in T}(1-\delta_p) \cdot  \prod_{\substack{p\in\fP\setminus T\\ p\nmid d}}\delta_p\\
&= \prod_{p\mid d}(1-\delta_p) \cdot\prod_{\substack{p\in \fP\\ p\nmid d}}(1-\delta_p+\delta_p)\\
&= \prod_{p\mid d}(1-\delta_p),
\end{align*}
from which and \eqref{eq:Z(emptyset)-upperbound} we conclude the upper-bound part of Theorem \ref{thm:sieve-densityversion}.
The lower-bound sieve part follows similarly.

\smallskip

\bibliographystyle{plain}

\end{document}